\pdfoutput=1
\documentclass[11pt,english]{article}

\usepackage[T1]{fontenc}
\usepackage[latin9]{inputenc}
\usepackage{geometry}
\usepackage{babel}
\usepackage{float}
\usepackage{amsthm}
\usepackage{amsmath}
\usepackage{amssymb}
\usepackage{setspace}
\usepackage{esint}
\usepackage[unicode=true,pdfusetitle,
 bookmarks=true,bookmarksnumbered=false,bookmarksopen=false,
 breaklinks=false,pdfborder={0 0 0},backref=false,colorlinks=false]
 {hyperref}

\usepackage[nameinlink,capitalise,noabbrev]{cleveref}
\AtBeginDocument{\let\oldref\ref \renewcommand{\ref}[1]{\cref{#1}}}

\theoremstyle{plain}
\newtheorem{thm}{\protect\theoremname}
\crefname{thm}{Theorem}{Theorems}
\theoremstyle{definition}

\theoremstyle{definition}

\theoremstyle{plain}

\theoremstyle{plain}

\theoremstyle{plain}
\newtheorem{lem}[thm]{\protect\lemmaname}
\crefname{lem}{Lemma}{Lemmas}
\theoremstyle{plain}

\theoremstyle{remark}
\newtheorem*{rem*}{\protect\remarkname}
\theoremstyle{remark}

\crefname{enumi}{Part}{Parts}

\crefformat{equation}{#2(#1)#3}

\let\originalleft\left
\let\originalright\right
\renewcommand{\left}{\mathopen{}\mathclose\bgroup\originalleft}
\renewcommand{\right}{\aftergroup\egroup\originalright}
\usepackage{pgfplots}
\usetikzlibrary{pgfplots.groupplots}
\usepackage{verbatim}

\makeatletter
\renewcommand*{\UrlTildeSpecial}{%
  \do\~{%
    \mbox{%
      \fontfamily{ptm}\selectfont
      \textasciitilde
    }%
  }%
}%
\let\Url@force@Tilde\UrlTildeSpecial
\makeatother

\usepackage{tikz}
\usetikzlibrary{decorations.markings}
\tikzstyle{vertex}=[circle,draw=black,fill=black,inner sep=0,minimum size=0.2cm,text=white,font=\footnotesize]
\tikzset{arc/.style={
  postaction={decorate}
}}
\tikzset{every loop/.style={min distance=50,in=50,out=130,looseness=7}}

\usetikzlibrary{shapes.geometric}
\tikzset{
  larrow/.style={
    shape border rotate=90,
    regular polygon,
    regular polygon sides=3,
    fill=black,
    inner sep=0,
    minimum height=2.3mm
  }
}
\tikzset{
  rarrow/.style={
    shape border rotate=-90,
    regular polygon,
    regular polygon sides=3,
    fill=black,
    inner sep=0,
    minimum height=2.3mm
  }
}

\usepackage[labelfont=bf,labelsep=period]{caption}

\usepackage{enumitem}

\makeatother

  \providecommand{\claimname}{Claim}
  \providecommand{\conjecturename}{Conjecture}
  \providecommand{\corollaryname}{Corollary}
  \providecommand{\definitionname}{Definition}
  \providecommand{\examplename}{Example}
  \providecommand{\lemmaname}{Lemma}
  \providecommand{\propositionname}{Proposition}
  \providecommand{\remarkname}{Remark}
\providecommand{\theoremname}{Theorem}

\date{}

\begin{document}

\title{\texorpdfstring{\vspace{-1cm}}{}Cycles and matchings in randomly perturbed digraphs and hypergraphs}

\author{Michael Krivelevich
\thanks{School of Mathematical Sciences, Raymond and Beverly Sackler Faculty of Exact Sciences, Tel Aviv University, 6997801, Israel. Email: \href{mailto:krivelev@post.tau.ac.il}
{\nolinkurl{krivelev@post.tau.ac.il}}. Research supported in part by USA-Israel BSF Grant 2010115 and
by grant 912/12 from the Israel Science Foundation.}\and
Matthew Kwan\thanks{Department of Mathematics, ETH, 8092 Z\"urich. Email: \href{mailto:matthew.kwan@math.ethz.ch}
{\nolinkurl{matthew.kwan@math.ethz.ch}}.}\and
Benny Sudakov\thanks{Department of Mathematics, ETH, 8092 Z\"urich. Email: \href{mailto:benjamin.sudakov@math.ethz.ch}{\nolinkurl{benjamin.sudakov@math.ethz.ch}}. Research supported in part by SNSF grant 200021-149111.}}

\maketitle
\global\long\def\RR{\mathbb{R}}
\global\long\def\QQ{\mathbb{Q}}
\global\long\def\HH{\mathbb{H}}
\global\long\def\E{\mathbb{E}}
\global\long\def\Var{\operatorname{Var}}
\global\long\def\CC{\mathbb{C}}
\global\long\def\NN{\mathbb{N}}
\global\long\def\ZZ{\mathbb{Z}}
\global\long\def\GG{\mathbb{G}}
\global\long\def\BB{\mathbb{B}}
\global\long\def\DD{\mathbb{D}}
\global\long\def\cL{\mathcal{L}}
\global\long\def\supp{\operatorname{supp}}
\global\long\def\one{\boldsymbol{1}}
\global\long\def\range#1{\left[#1\right]}
\global\long\def\d{\operatorname{d}}
\global\long\def\falling#1#2{\left(#1\right)_{#2}}
\global\long\def\f{\mathbf{f}}
\global\long\def\im{\operatorname{im}}
\global\long\def\sp{\operatorname{span}}
\global\long\def\sign{\operatorname{sign}}
\global\long\def\mod{\operatorname{mod}}
\global\long\def\id{\operatorname{id}}
\global\long\def\disc{\operatorname{disc}}
\global\long\def\lindisc{\operatorname{lindisc}}
\global\long\def\tr{\operatorname{tr}}
\global\long\def\adj{\operatorname{adj}}
\global\long\def\Unif{\operatorname{Unif}}
\global\long\def\Po{\operatorname{Po}}
\global\long\def\Bin{\operatorname{Bin}}
\global\long\def\Ber{\operatorname{Ber}}
\global\long\def\Pr{\mathbb{P}}
\global\long\def\Geom{\operatorname{Geom}}
\global\long\def\Hom{\operatorname{Hom}}
\global\long\def\floor#1{\left\lfloor #1\right\rfloor }
\global\long\def\ceil#1{\left\lceil #1\right\rceil }
\global\long\def\T{\mathbf{T}}
\global\long\def\N{\mathbf{N}}
\global\long\def\B{\mathbf{B}}
\global\long\def\cdi{c}
\global\long\def\chyp{c_{k}}
\global\long\def\chypham{\chyp}
\global\long\def\chypmat{\chyp}
\global\long\def\cto{c^{\mathrm{T}}}
\global\long\def\a{\beta_k}
\global\long\def\amat{\a}
\global\long\def\aham{\a}

\begin{abstract}
We give several results showing that different discrete structures typically
gain certain spanning substructures (in particular, Hamilton cycles) after a modest random perturbation. First, we prove
that adding linearly many random edges to a dense $k$-uniform hypergraph ensures
the (asymptotically almost sure) existence of a perfect matching or a loose
Hamilton cycle. The proof involves an interesting application of Szemer\'edi's
Regularity Lemma, which might be independently useful. We next prove that digraphs with certain strong expansion
properties are pancyclic, and use this to show that adding a linear
number of random edges typically makes a dense digraph pancyclic. Finally, we
prove that perturbing a certain (minimum-degree-dependent) number of
random edges in a tournament typically ensures the existence of multiple edge-disjoint
Hamilton cycles. All our results are tight.


\end{abstract}

\section{Introduction and Results}

We say that a graph is \emph{Hamiltonian} if it has a \emph{Hamilton
cycle}: a simple cycle containing every vertex in the graph. Hamiltonicity
is a central notion in graph theory and has been extensively studied
in a wide range of contexts. In particular, due to a seminal paper
by Karp \cite{Kar72}, it has become a canonical NP-complete problem
to determine whether an arbitrary graph is Hamiltonian. There are
nevertheless a variety of easily-checkable conditions that guarantee
Hamiltonicity. The most famous of these is given by a classical theorem
of Dirac \cite{Dir52}, which states that any $n$-vertex graph ($n\ge 3$) with
minimum degree at least $n/2$ is Hamiltonian.

Dirac's theorem demands a very strong density condition, but in a
certain asymptotic sense ``almost all'' dense graphs are Hamiltonian.
If we fix $\alpha>0$ and select a graph uniformly at random among
the (labelled) graphs with $n$ vertices and $\alpha{n \choose 2}$
edges, then the degrees will probably each be about $\alpha n$. Such
a random graph is Hamiltonian with probability $1-o\left(1\right)$
(we say it is Hamiltonian \emph{asymptotically almost surely}, or
\emph{a.a.s.}). This follows from a stronger result independently due to P\'osa \cite{Pos76} and Korshunov \cite{Kor76}
that gives a \emph{threshold} for Hamiltonicity: a random $n$-vertex,
$m$-edge graph is Hamiltonian a.a.s.{} if $m\gg n\log n$, and fails
to be Hamiltonian a.a.s.{} if $m\ll n\log n$. Here and from now on,
all asymptotics are as $n\to\infty$, and we implicitly round large
quantities to integers.

In \cite{BFM03}, Bohman, Frieze and Martin studied the random graph model that
starts with a dense graph and adds $m$ random edges (this model has
since been studied in a number of other contexts; see for example
\cite{BHM04,KST06}). They found that to ensure Hamiltonicity in this
model we only need $m$ to be linear, saving a logarithmic factor
over the standard model where we start with nothing. To be precise,
\cite[Theorem~1]{BFM03} says that for every $\alpha>0$ there is
$\cdi=\cdi\left(\alpha\right)$ such that if we start with a graph with
minimum degree at least $\alpha n$ and add $\cdi n$
random edges, then the resulting graph will a.a.s.{} be Hamiltonian.
Note that some dense graphs require a linear number of extra edges
to become Hamiltonian (consider the complete bipartite graph with
partition sizes $n/3$ and $2n/3$), so the order of magnitude of this result is tight. 
We can interpret this theorem as quantifying the ``fragility''
of the few dense graphs that are not Hamiltonian, by determining the
amount of random perturbation that is necessary to make a dense graph
Hamiltonian. A comparison can be drawn to the notion of \emph{smoothed
analysis} of algorithms introduced by Spielman and Teng \cite{ST04}, which involves
studying the performance of algorithms on randomly perturbed inputs.

Our first contribution in this paper is to generalize the aforementioned
theorem to hypergraphs (and to give a corresponding result for perfect
matchings, which is nontrivial in the hypergraph setting). Unfortunately,
there is no single most natural notion of a cycle or of minimum degree
in hypergraphs. A $k$-uniform \emph{loose} cycle is a $k$-uniform
hypergraph with a cyclic ordering on its vertices such that every
edge consists of $k$ consecutive vertices and every pair of consecutive
edges intersects in exactly one vertex. The degree of a set of vertices
is the number of edges that include that set, and the \emph{minimum
$\left(k-1\right)$-degree }$\delta_{k-1}$ is the minimum degree
among sets of $k-1$ vertices. Let $\HH_{k}\left(n,m\right)$ be
the uniform distribution on $m$-edge $k$-uniform hypergraphs on
the vertex set $\range n$.
\begin{thm}
\label{thm:hypergraph-theorems}For each $\alpha>0$ there is $\chyp=\chyp\left(\alpha\right)$ such that:

\begin{enumerate}[topsep=0px,label=(\alph*)]

\item{\label{itm:hypergraph-matching-theorem}If $H$ is
a $k$-uniform hypergraph on $\range{kn}$ with $\delta_{k-1}\left(H\right)\ge\alpha n$,
and $R^{\mathrm{M}}\in\HH_{k}\left(kn,\chypmat n\right)$, then
$H\cup R$ a.a.s.{} has a perfect matching.}

\item{\label{itm:hypergraph-cycle-theorem}If $H$ is
a $k$-uniform hypergraph on $\range{\left(k-1\right)n}$with
$\delta_{k-1}\left(H\right)\ge\alpha n$, and $R\in\HH\left(\left(k-1\right)n,\chypham n\right)$,
then $H\cup R$ a.a.s.{} has a loose Hamilton cycle.}

\end{enumerate}
\end{thm}
All the motivation for graphs is still relevant in the hypergraph setting. Dirac's theorem approximately generalizes to hypergraphs (see \cite{KKMO11}): for small $\varepsilon$ and large $n$, if the minimum $\left(k-1\right)$-degree of an $n$-vertex $k$-uniform hypergraph is greater than $\left(1/\left(2\left(k-1\right)\right)+\varepsilon\right)\,n$ then that hypergraph contains a loose Hamilton cycle. Just as for graphs, the threshold for both perfect matchings and loose Hamilton cycles in $k$-uniform hypergraphs is $n\log n$ random
edges (see \cite{DF11} and \cite[Corollary~2.6]{JKV08}), so ``almost
all'' dense hypergraphs have Hamilton cycles and perfect matchings.

We will prove \ref{thm:hypergraph-theorems}, and show that it is
tight, in \ref{sec:hypergraphs}. The methods usually employed to
study Hamilton cycles and perfect matchings in random graphs are largely
ineffective in the hypergraph setting, so we need a very different
proof. In particular, we cannot easily manipulate paths for P\'osa-type
arguments, and we do not have an analogue of Hall's marriage theorem
allowing us to deduce the existence of a perfect matching from an
expansion property. Our proof involves reducing the theorem to the
existence of a perfect matching in a certain randomly perturbed dense bipartite
graph. The na\"ive approach to prove the existence of a perfect matching in this graph would be to use Hall's theorem and the union bound. Unfortunately this fails, and in fact the ``reason'' for a perfect matching in this perturbed graph seems to be quite different depending on the structure of the initial bipartite graph. The proof therefore makes use of the structural description provided by Szemer\'edi's regularity lemma in an interesting way.

Our second contribution in this paper is a theorem giving a general
expansion condition for\emph{ pancyclicity}. We say an $n$-vertex
\mbox{(di-)}graph is pancyclic if it contains cycles of all lengths
ranging from 3 to $n$.
\begin{thm}
\label{lem:pseudorandom-pancyclic}Let $D$ be a directed graph on
$n$ vertices with all in- and out- degrees at least $8k$, and suppose
for every pair of disjoint sets $A,B\subseteq V\left(D\right)$ with $\left|A\right|=\left|B\right|\ge k$,
there is an arc from $A$ to $B$. Then $D$ is pancyclic.
\end{thm}
We hope this theorem could be of independent interest, but our particular
motivation is that it implies a number of results about randomly perturbed
graphs and digraphs. In particular it provides very simple proofs
of the theorems in \cite{BFM03} concerning Hamiltonicity in randomly
perturbed graphs and digraphs, and allows us to extend these theorems
to pancyclicity. Most generally, \ref{lem:pseudorandom-pancyclic} implies the following theorem. Let $\DD\left(n,m\right)$ be the uniform
distribution on $m$-arc digraphs on the vertex set $\range n$ (in this paper we allow 2-cycles in digraphs, so there are $2{n\choose2}$ possible arcs).

\begin{thm}
\label{thm:smoothed-pancyclic}For each $\alpha>0$, there is $\cdi=\cdi\left(\alpha\right)$ such that
if $D$ is a digraph on $\range n$ with all in- and out- degrees at least $\alpha n$,
and $R\in\DD\left(n,\cdi n\right)$, then $D\cup R$
is a.a.s.{} pancyclic.
\end{thm}

We will prove \ref{lem:pseudorandom-pancyclic,thm:smoothed-pancyclic} in  \ref{sec:digraphs}.

Our final theorem concerns randomly perturbed tournaments. The model
that starts with a fixed \mbox{(di-/hyper-)}graph and adds random
edges is not suitable for studying random perturbation in tournaments,
because we want our perturbed tournament to remain a tournament. There
are several other models of random perturbation we could consider
that do allow us to make sense of randomly perturbed tournaments,
or are more natural in certain contexts. However, the types of results
in this paper are not sensitive to the model used. We will briefly
describe a few different models here.

First, note that for most practical purposes, models that involve
the selection of $m$ random edges are equivalent to models that involve
the selection of each edge with probability $p$ independently, where
$m=pN$ and $N$ is the total number of possible edges. One perspective or the other can be more intuitive
or result in cleaner proofs; we will use both interchangeably as is
convenient without further discussion. In all the situations in this
paper, equivalence can be proved with standard conditioning and coupling
arguments (see for example \cite[Section~1.4]{JLR00}).

As suggested by Spielman and Teng in \cite[Definition~1]{ST03}, one possible alternative
model is to \emph{change} random edges, instead of adding them. So,
for our results so far, instead of taking the union of a fixed dense
object with a random sparse object, we would take the symmetric difference.
Our results still hold in this alternative model, basically because
we can break up such a random perturbation into a phase that deletes
edges (this does not destroy denseness), and a phase that adds edges.
One undesirable quirk of this model is that it is not ``monotonic'':
if we change too many edges then we ``lose our randomness'' and
end up at the complement of our original object.

A second alternative model is to start with our fixed object and ``make
it more random'' by interpolating slightly towards the corresponding
uniform distribution. For example, in the graph case we could randomly
designate a small number of pairs of vertices for ``resampling''
and then decide whether the corresponding edges should be present
uniformly and independently at random. This is mostly equivalent to
the symmetric difference model, and is the model in which we prefer
to state our theorem about randomly perturbed tournaments.

Although it is easy to construct tournaments with no Hamilton cycle, here we prove 
that every tournament becomes Hamiltonian after a small random perturbation. We also
show that randomly perturbed tournaments are not just Hamiltonian,
but have multiple edge-disjoint Hamilton cycles.
Moreover, we can give stronger results
for tournaments with large minimum in- and out- degrees. 
Recall that $\omega\left(f\right)$ represents a function that grows faster than 
$f$ (to be precise, $g=\omega\left(f\right)$ means $g/f\to\infty$).
\begin{thm}
\label{thm:tournament}Consider a tournament $T$ with $n$ vertices
and all in- and out- degrees at least $d$. Independently choose
$m=\omega\left(n/\left(d+1\right)\right)$ random edges of $T$ and orient them uniformly at random.
The resulting perturbed tournament $P$ a.a.s.{} has $q$ arc-disjoint
Hamilton cycles, for $q=O\left(1\right)$.
\end{thm}

Note that we allow for the case where the minimum degree $d$ is zero, and where $d$ is an arbitrary function of $n$. We will prove \ref{thm:tournament}, and show that it is tight, in \ref{sec:tournaments}.

\section{\label{sec:hypergraphs}Perfect Matchings and Hamilton Cycles in
Hypergraphs}

We first make some observations about our minimum degree requirement.
The \emph{minimum $q$-degree} $\delta_{q}\left(H\right)$ of a hypergraph $H$
is the minimum degree among all sets of $q$ vertices. Note that this
generalizes the two notions of denseness for graphs: in some contexts,
we say graphs are dense if they have many edges,
whereas in this paper we need a stronger notion of graph denseness
based on minimum degree. For a $k$-uniform hypergraph
$H$, a double-counting argument shows that if $q\le p$ then
\[
\delta_{q}\left(H\right)\ge\delta_{p}\left(H\right){n-q \choose p-q}\left.\vphantom{\sum}\right/{k-q \choose p-q}.
\]
So, imposing that a $k$-uniform hypergraph has large $\left(k-1\right)$-degree
ensures that it has large $q$-degrees for all $q$. In particular,
our requirement $\delta_{k-1}\left(H\right)=\Omega\left(n\right)$
actually implies $\delta_{q}\left(H\right)=\Omega\left(n^{k-q}\right)$
for all $q$.

Next, note that \ref{thm:hypergraph-theorems} is tight for essentially
the same reason as its corresponding theorem for graphs. Consider the dense ``complete
bipartite hypergraph'' which has two parts of sizes $n$ and
$2kn$, and has all possible $k$-edges that contain at least one
vertex from each part. Only $2n$ of these edges can contribute
to a loose Hamilton cycle, so a linear number must be added to complete
the necessary $\left(2k+1\right)n/\left(k-1\right)$ edges. Similarly, this graph contains only $n$ out of the $\left(2k+1\right)n/k$ required edges in a perfect matching.

Now we proceed to the proof of \ref{thm:hypergraph-theorems}, which
will follow from a sequence of lemmas. We will assume $k\ge3$, since the case $k=2$ is proved in \cite{BFM03}. The first step is to show that
$R$ almost gives the structure of interest on its own. Let a \emph{partial-cycle} be a hypergraph which can be extended to a loose Hamilton
cycle by adding edges. Recall that $\HH_{k}\left(n,m\right)$ is
the uniform distribution on $m$-edge $k$-uniform hypergraphs on
the vertex set $\range n$.

\global\long\def\b{\chyp}
\global\long\def\bmat{\b}
\global\long\def\bham{\b}
\global\long\def\patheps{\gamma}

\begin{lem}
\label{lem:hypergraph-almost-object}For any $\varepsilon>0$, there is $\b=\b\left(\varepsilon\right)$ such that:

\begin{enumerate}[topsep=0px,label=(\alph*)]

\item{\label{itm:hypergraph-large-matching}$R^{\mathrm{M}}\in\HH_{k}\left(kn,\bmat n\right)$
a.a.s.{} has a matching of $\left(1-\varepsilon\right)n$ edges.}

\item{\label{itm:hypergraph-large-subcycle}$R^{\mathrm{H}}\in\HH_{k}\left(\left(k-1\right)n,\bham n\right)$
a.a.s.{} has a partial-cycle with $\left(1-\varepsilon\right)n$ edges.}

\end{enumerate}\end{lem}
\begin{proof}
First we define a \emph{loose path} by analogy with loose cycles:
a $k$-uniform loose path is a hypergraph with a vertex ordering such
that every edge consists of $k$ consecutive vertices and every pair
of consecutive edges intersects in exactly one vertex. We say that the edges at the start and at the end of the ordering are \emph{extremal edges}. A matching
is a collection of loose paths of length 1, and a partial-cycle is any
collection of loose paths with enough vertices left over to link them
together into a cycle.

We will first prove the following:

\textbf{Claim:} for any $\ell\ge1$ and $\patheps>0$, there
is $h=h\left(\ell,\patheps\right)$ such that the following holds. In a
random hypergraph $R\in\HH_{k}\left(kn,hn\right)$,
a.a.s.{} every set of $\patheps n$ vertices contains a loose path of length
$\ell$.

Parts \oldref{itm:hypergraph-large-matching} and \oldref{itm:hypergraph-large-subcycle} of the Lemma will follow from this claim.

In accordance with the discussion in the introduction, it is equivalent to consider
the model for $R$ where each edge is independently present with probability
$p=hn/{kn \choose k}=O\left(n^{1-k}\right)$.

The probability a particular pair of edges is present in $R$ is $p^2$. There are $O\left(n^2\right)$ pairs of vertices and for each there are $O\left({n-2\choose k-2}^2\right)=
O\left(n^{2\left(k-2\right)}\right)$ pairs of edges containing both those vertices. So, the expected number of pairs of vertices
which are contained in more than one edge (have degree more than one)
is $O\left(n^2n^{2\left(k-2\right)}p^2\right)=O\left(1\right)$. By Markov's inequality, there are a.a.s.{} fewer
than $\patheps n/2$ such pairs of vertices. So if we delete a
set $D$ containing one member of each of those pairs, then every pair of vertices in the remaining hypergraph has
degree at most one.

Let $d=k\ell$. For
large $h$, using the Chernoff bound together with the union
bound, it is easy to show that a.a.s.{} every set of $\patheps n/2$ vertices spans at least $d\patheps n/2$
edges, which is to say that the average 1-degree in the induced subhypergraph
is at least $kd$. We assume this holds for the remainder of the proof.

Every set $S$ of $\patheps n$ vertices includes a set of $\patheps n/2$
vertices disjoint from $D$, which has average 1-degree at least $kd$.
Deleting a vertex of degree less than $d$ increases the average degree
of the induced subhypergraph, so $S\backslash D$ includes a set of
vertices spanning a subhypergraph $Q$ of $R$ with minimum degree
at least $d$. Let $P$ be a longest loose path in $Q$ and let $v$
be a vertex with degree one in $P$, in one of the extremal edges of $P$. Since $P$ cannot
be extended to a longer path, each of the (at least $d$) edges containing
$v$ also contains another vertex $u$ of $P$. But because $Q$ contains
no vertices from $D$, there is at most
one edge containing both $v$ and $u$, so $P$ must have at least
$d$ vertices and therefore has length at least $\ell$. This proves
our claim.

We now prove \ref{itm:hypergraph-large-matching}. Consider a matching of maximum size in $R^{\mathrm{M}}$.
There can be no edge consisting of unmatched vertices because this would
allow us to extend the matching, contradicting maximality. Applying
our claim with $\patheps=k\varepsilon$ and $\ell=1$, we can see that
if $\bmat$ is large enough then
a.a.s.{} every set of $k\varepsilon n$ vertices spans at least one
edge in $R$. This proves there are fewer than $k\varepsilon n$
vertices unmatched after our maximum matching, hence our matching
has at least $\left(1-\varepsilon\right)n$ edges.

Finally we prove \ref{itm:hypergraph-large-subcycle}, using roughly the same idea. It takes $\left(k-2\right)q$ additional
vertices to link $q$ loose paths together into a loose cycle, and
a union of $q$ disjoint loose paths of length $\ell\ge1$ has $\left(\ell\left(k-1\right)+1\right)q$
vertices. So, such a union of paths is a partial-cycle in $R^{\mathrm{H}}$
precisely when 
\[
\left(k-2\right)q+\left(\ell\left(k-1\right)+1\right)q\le\left(k-1\right)n,
\]
which simplifies to the condition $q\left(\ell+1\right) \leq n$.

We will apply our claim with $\ell=1/\varepsilon$ and $\patheps=\left(k-2\right)/\left(\ell+1\right)$. Consider a maximum-size collection of disjoint length-$\ell$ loose paths in $R^{\mathrm{H}}$. Our claim proves that if $\bmat$ is large enough then there are fewer than $\left(k-2\right)n/\left(\ell+1\right)$ vertices left over after our maximal collection of loose paths. This means our maximal collection has at least
\[
\frac{\left(k-1\right)n-\left(k-2\right)n/\left(\ell+1\right)}{\ell\left(k-1\right)+1}= n/\left(\ell+1\right)=:q
\]
loose paths. Since $q\left(\ell+1\right) \leq n$, a subcollection of $q$ of these loose paths gives a partial-cycle, which has $\ell q=\left(1-1/\left(\ell+1\right)\right)n>\left(1-\varepsilon\right)n$ edges.
\end{proof}
The second step to prove \ref{thm:hypergraph-theorems} is to show that a dense
hypergraph plus a large partial structure a.a.s.{} gives the structure
we are looking for. For both theorems, we will be able to reduce this
step to the following lemma.
\begin{lem}
\label{lem:bipartite-plus-big-matching-perfect}There is $\xi=\xi\left(\alpha\right)>0$ 
such that the following holds. Let $G$ be a bipartite graph with
parts $A,B$ of equal size $n$, and suppose $\delta\left(G\right)\ge\alpha n$.
Let $\bar{M}$ be a uniformly random perfect matching between $A$
and $B$ (not necessarily contained in $G$) and let $M$ be any sub-matching of $\bar{M}$ with $\left(1-\xi\right) n$
edges. Then a.a.s.{} $G\cup M$ has a perfect matching. 
\end{lem}

Note that while we require $\bar M$ to be uniformly random, we make no assumptions about the distribution of $M$ other than that is contained within $\bar M$.

The immediate na\"ive approach to prove this lemma would be to show each set of vertices expands, and then to apply the union bound and Hall's marriage theorem. However, the
probability of failure to expand is not small enough for this to work.
We can gain some insight into the problem by considering two ``extremal''
cases for $G$. First, consider the case where the edges of $G$ are
not evenly-distributed, and are ``concentrated'' in certain spots.
For example, identify sets $A'\subset A$ and $B'\subset B$ with
$\left|A'\right|,\left|B'\right|=\alpha n$, and let $G$ contain
only those edges incident to a vertex in $A'$ or $B'$. The addition
of the near-perfect matching $M$ gives a near-perfect matching between
$A\backslash A'$ and $B\backslash B'$, and we can match the unmatched
vertices from $A$ (respectively $B$) with any element of $B'$ (respectively
$A'$). That is, if our graph is not well-distributed, then the more
concentrated parts help us to augment $M$ into a perfect matching in
$G\cup M$. On the other extreme, if $G$ is a random-like, well-distributed
dense graph then we cannot augment $M$ in the same way. But this
is not necessary, because a random dense graph $G$ contains a perfect
matching on its own! In order to apply these ideas to prove the lemma
for an arbitrary graph $G$, we use the structural description of
$G$ provided by Szemer\'edi's regularity lemma.

For a disjoint pair of vertex sets $\left(X,Y\right)$ in a graph,
let its \emph{density} $d\left(X,Y\right)$ be the number of edges
between $X$ and $Y$, divided by $\left|X\right|\left|Y\right|$.
A pair of vertex sets $\left(V^{1},V^{2}\right)$ is said to be \emph{$\varepsilon$-regular}
if for any $U^{1},U^{2}$ with $U^{\ell}\subseteq V^{\ell}$ and $\left|U^{\ell}\right|\ge\varepsilon\left|V^{\ell}\right|$,
we have $\left|d\left(U^{1},U^{2}\right)-d\left(V^{1},V^{2}\right)\right|\le\varepsilon$.

\global\long\def\s{s}
\global\long\def\k{r}
\global\long\def\K{K}

We will use a bipartite version of the regularity lemma (which can
be deduced from say \cite[Theorem~2.3]{Tao05} in a similar way to
\cite[Theorem~1.10]{KS96}). Let $\alpha'=\alpha/2$ and let $\varepsilon>0$
be a small constant depending on $\alpha$ that will be determined
later (assume for now that $\varepsilon<\alpha/8$). There is a large
constant $\K$ depending only on $\alpha$ such that there exist partitions
$A=V_{0}^{1}\cup\dots\cup V_{\k}^{1}$ and $B=V_{0}^{2}\cup\dots\cup V_{\k}^{2}$
with $\k\le \K$, in such a way that the following conditions are satisfied.
The ``exceptional'' clusters $V_{0}^{1}$ and $V_{0}^{2}$ both
have fewer than $\varepsilon n$ vertices, and the non-exceptional
clusters in $A$ and $B$ have equal size: $\left|V_{i}^{\ell}\right|=\s n$. (Note that this implies $1-\varepsilon\le rs\le1$).
There is a subgraph $G'\subseteq G$ with minimum degree at least
$\left(\alpha'+\varepsilon\right)n$ such that each pair of distinct
clusters $V_{i}^{1},V_{j}^{2}$ ($i,j\ge1$) is $\varepsilon$-regular
in $G'$ with density zero or at least $2\varepsilon$.

Define the cluster graph $\Gamma$ as the bipartite graph whose vertices
are the non-exceptional clusters $V_{i}^{\ell}$, and whose edges
are the pairs of clusters between which there is nonzero density in
$G'$. The fact that $G'$ is dense implies that $\Gamma$ is dense as well, as follows. In $G'$ each $V_{i}^{\ell}$ has at least $\left(\alpha'+\varepsilon\right)n\left|V_{i}^{\ell}\right|$
edges to other clusters. There are at most $\left(\varepsilon n\right)\left(\s n\right)$
edges to the exceptional cluster $V_{0}$ and at most $\left(\s n\right)^{2}$
edges to each other cluster. So, $d_{\Gamma}\left(V_{i}^{\ell}\right)\ge\left(\left(\alpha'+\varepsilon\right)n-\varepsilon n\right)\s n/\left(\s n\right)^{2}\ge\alpha'\k$
and $\Gamma$ has minimum degree at least $\alpha'\k$.
\begin{proof}[Proof of \ref{lem:bipartite-plus-big-matching-perfect}]
We use Hall's marriage theorem: we need to show that a.a.s.{} $\left|N_{G\cup M}\left(W\right)\right|\ge\left|W\right|$
for all $W\subseteq A$. If $\left|W\right|\le\alpha n$, then $\left|N_{G\cup M}\left(W\right)\right|\ge\left|N_{G}\left(W\right)\right|\ge\alpha n\ge\left|W\right|$
by the degree condition on $G$. Similarly, if $\left|W\right|\ge\left(1-\alpha\right)n$
then every $b\in B$ has an edge to $W$ in $G$, so $\left|N_{G\cup M}\left(W\right)\right|=\left|B\right|\ge\left|W\right|$.
The difficult case is where $\alpha n\le\left|W\right|\le\left(1-\alpha\right)n$.

Apply the bipartite version of the regularity lemma above (with $\varepsilon$ sufficiently small compared to $\alpha$), to obtain
a subgraph $G'$ and a dense cluster graph $\Gamma$. Consider any $W\subset A$ with $\alpha n\le\left|W\right|\le\left(1-\alpha\right)n$.
For each $i$ let $\pi_{i}=\left|V_{i}^{1}\cap W\right|/\left(\s n\right)$,
and let $X$ be the set of clusters $V_{i}^{1}$ with $\pi_{i}\ge\varepsilon$. Let $A_X$ denote the set of all vertices in those clusters.
Note that $\left|W\backslash A_X\right|\le r\varepsilon sn\le\varepsilon n$. Also,
if $\varepsilon$ is small compared to $\alpha$, then $X$ must be
nonempty. Now, if $V_{j}^{2}\in N_{\Gamma}\left(X\right)$ then by $\varepsilon$-regularity
there are edges in $G'$ from $W$ to at least $\left(1-\varepsilon\right)\s n$
vertices of $V_{j}^{2}$. Let $Y=N_{\Gamma}\left(X\right)$ and let $B_Y$ be the set of vertices in the clusters in $Y$; it follows that $\left|B_Y\backslash N_{G'}\left(W\right)\right|\le\varepsilon\left|Y\right|sn\le\varepsilon n$.

If $\left|Y\right|=\k$ (as would occur if $G$
was well-distributed) then
\[
\left|N_{G\cup M}\left(W\right)\right|\ge\left|N_{G'}\left(W\right)\right|\ge \left|B_Y\right|-\varepsilon n\ge n-2\varepsilon n\ge\left|W\right|
\]
for $2\varepsilon\le \alpha$, and we are done.

Otherwise, there is $V_{j}^{2}$ outside $Y=N_{\Gamma}\left(D\right)$. This
$V_{j}^{2}$ must have $\alpha'\k$ neighbours outside $X$ in $\Gamma$,
so $\left|X\right|\le\left(1-\alpha'\right)\k$. Now, $Y\ge\alpha'r$ so $B_Y$ is a fixed set of at least $\alpha'rsn\ge \left(1-\varepsilon\right)\alpha'n$
vertices in $B$. Also, $N_{\bar{M}}\left(A_X\right)$ is a uniformly
random set of $\left|A_X\right|\le\left(1-\alpha'\right)rsn\le\left(1-\alpha'\right)n$
vertices in $B$, which means $B\backslash N_{\bar{M}}\left(A_X\right)$
is a uniformly random set of at least $\alpha'n$ vertices. So, $\left|B_Y\backslash N_{\bar{M}}\left(A_X\right)\right|$
is hypergeometrically distributed with mean at least $\left(1-\varepsilon\right)\left(\alpha'\right)^{2}n$.
By a concentration inequality (see for example \cite[Theorem~2.10]{JLR00}),
a.a.s.
\[
\left|B_Y\backslash N_{\bar{M}}\left(A_X\right)\right|\ge\left(\alpha'\right)^{2}n-2\varepsilon n.
\]
There are fewer than $2^r=O(1)$ possibilities for $X$, so by the union bound this inequality holds a.a.s.{} for the $X$ arising from \emph{any} choice of $W$. Now, note that $\left|N_{\bar{M}}\left(W\right)\backslash N_{\bar{M}}\left(A_X\right)\right|=\left|W\backslash A_X\right|$ and recall that $\left|W\backslash A_X\right|,\,\left|B_Y\backslash N_{G'}\left(W\right)\right|\le\varepsilon n$, so
\[
\left|N_{G'}\left(W\right)\backslash N_{\bar{M}}\left(W\right)\right|
  \ge\left|B_Y\backslash N_{\bar{M}}\left(A_X\right)\right|-2\varepsilon n\ge\left(\alpha'\right)^{2}n-4\varepsilon n.
\]
Also, $\left|N_{M}\left(W\right)\right|\le \left|N_{\bar M}\left(W\right)\right|-\xi n=\left|W\right|-\xi n$. We conclude that
\[
\left|N_{G\cup M}\left(W\right)\right| \ge\left|N_{M}\left(W\right)\right|+\left|N_{G'}\left(W\right)\backslash N_{\bar{M}}\left(W\right)\right|
  \ge\left|W\right|-\xi n+\left(\alpha'\right)^{2}n-4\varepsilon n.\\
\]
For small $\varepsilon$ and $\xi$ (say $4\varepsilon\le\xi=\left(\alpha'\right)^{2}/2=\alpha^2/8$), this
gives $\left|N_{G\cup M}\left(W\right)\right|\ge \left|W\right|$.
\end{proof}
\begin{rem*}
Note that $\xi$ does not depend on any of the constants arising in the regularity lemma. These constants only influence the value of $n$ needed to make the probability implicit in ``a.a.s.'' close to 1.
\end{rem*}
Now we describe the reduction of \ref{thm:hypergraph-theorems} to
\ref{lem:bipartite-plus-big-matching-perfect}. Consider a $k$-uniform
hypergraph $L$. Suppose $A$ is a set of $n$ vertices and $B$ is
a $\left(k-1\right)$-uniform hypergraph on the remaining vertices.
Then we define a bipartite graph $G_{A,B}\left(L\right)$ as follows.
The vertices of $G_{A,B}\left(L\right)$ are the vertices in $A$,
as well as the edges in $B$ (we abuse notation and identify the hypergraph
$B$ with its edge set). We put an edge between $a\in A$ and $\left\{ b_{1},\dots,b_{k-1}\right\} \in B$
if $\left\{ a,b_{1},\dots,b_{k-1}\right\} $ is an edge in $L$. Basically,
the idea is that if $L$ has a large matching or partial-cycle, then there
is $A$ and $B$ such that $G_{A,B}\left(L\right)$ has a large matching. Conversely,
if $G_{A,B}\left(L\right)$ has a large matching for any $A$
and $B$, then the edges of that matching correspond to a large matching or partial-cycle in $L$. \ref{lem:hypergraph-almost-object}
will provide a large matching or partial-cycle in the random hypergraph $R$, so there are $A$
and $B$ such that $G_{A,B}\left(R\right)$ has a large matching (this matching is itself random).
By \ref{lem:bipartite-plus-big-matching-perfect}, the addition of
this matching to $G_{A,B}\left(H\right)$ will give a perfect
matching in $G_{A,B}\left(H\right)\cup G_{A,B}\left(R\right)=G_{A,B}\left(H\cup R\right)$,
corresponding to a perfect matching or loose Hamilton cycle in $H\cup R$.

We make this precise as follows. For any $\varepsilon>0$ depending on $\alpha$, if $\chyp\left(\alpha\right)$
is large enough, \ref{lem:hypergraph-almost-object} ensures the a.a.s.{}
existence of a $\left(1-\varepsilon\right)n$-edge matching or partial-cycle $Q$ in $R$. Extend this to a perfect matching or loose Hamilton cycle
$\bar{Q}$ on $V\left(R\right)$ in an arbitrary way. Note that the distribution of $R$
is invariant under relabelling of its vertices so we can relabel its
vertices uniformly at random to assume that $\bar{Q}$ is a uniformly
random perfect matching or loose Hamilton cycle.

We describe an alternative way to realize a uniformly random perfect
matching or loose Hamilton cycle. In the perfect matching case, choose a
uniformly random ordering of $V\left(R\right)$: 
\[
a^{1},\dots,a^{n},\, b_{1}^{1},\dots,b_{k-1}^{1},\, b_{1}^{2},\dots,b_{k-1}^{n}.
\]
Let $b^{i}=\left\{ b_{1}^{i},\dots b_{k-1}^{i}\right\}$, and let our perfect matching have edges of the form $\left\{a^{i}\right\}\cup b^i $.
This gives a uniformly random perfect matching, so we can couple our
random ordering with $R$ in such a way that the perfect matching
defined by the ordering coincides with $\bar{Q}$. Then, let $A$
consist of the vertices $a^{i}$ and let $B$ be the hypergraph with
edges $b^i$.
Note that $A$ is a uniformly random $n$-vertex set with a uniformly random ordering $a^{1},\dots,a^{n}$, and $B$ is a
uniformly random $\left(k-1\right)$-uniform perfect matching on the remaining vertices. Also,
if we condition on $A$ and $B$ then $G_{A,B}\left(\bar{Q}\right)$
is a uniformly random perfect matching, and $G_{A,B}\left(Q\right)$
is a sub-matching with $\left(1-\varepsilon\right)n$ of its edges.

The Hamilton cycle case is similar. Again we want to define a uniformly random loose Hamilton cycle via a random ordering of $V\left(R\right)$. So, choose a uniformly random ordering:
\[
a^{1},\dots,a^{n},b_{0},\dots,b_{\left(k-2\right)n-1}.
\]
For each $i$ let $b^{i}$ be the $\left(k-1\right)$-vertex set $\left\{ b_{\left(k-2\right)i},b_{\left(k-2\right)i+1},\dots,b_{\left(k-2\right)\left(i+1\right)}\right\}$ (where the subscripts are interpreted modulo $\left(k-2\right)n$). Note that in this case consecutive $b^i$ intersect each other in one vertex. We define our loose Hamilton cycle to have edges of the form $\left\{ a^{i}\right\} \cup b^{i}$. This loose Hamilton cycle is uniformly random, so we can couple our random ordering with $R$ in such a way that our loose Hamilton cycle coincides with $\bar Q$. Let $A$ contain
the vertices $a^{i}$ and let $B$ contain the
edges $b^{i}$. Exactly the same considerations hold: $A$ is a uniformly
random $n$-vertex set with a uniformly random ordering $a^{1},\dots,a^{n}$, $B$ is a uniformly random $\left(k-1\right)$-uniform loose Hamilton
cycle on the remaining vertices, and $G_{A,B}\left(Q\right)$ is a
$\left(1-\varepsilon\right)n$-edge sub-matching of the uniformly random perfect
matching $G_{A,B}\left(\bar{Q}\right)$.

We give one final lemma, establishing that $G_{A,B}\left(H\right)$ is a.a.s.{} dense if $H$ is.
\begin{lem}
\label{lem:hypergraph-bipartite-degree-transfer}There is $\a=\a\left(\alpha\right)>0$ such that the following
holds.

\begin{enumerate}[topsep=0px,label=(\alph*)]

\item{\label{itm:hypergraph-bipartite-matching-degree-transfer}Let $H$ satisfy the conditions of \ref{thm:hypergraph-theorems}\oldref{itm:hypergraph-matching-theorem},
let $A$ be a uniformly random set of $n$ vertices, and let $B$
be a uniformly random perfect matching on $V\left(H\right)\backslash A$.
Then a.a.s.{} $G_{A,B}\left(H\right)$ has minimum degree at least $\amat n$.}

\item{\label{itm:hypergraph-bipartite-cycle-degree-transfer}Let $H$ satisfy the conditions of \ref{thm:hypergraph-theorems}\oldref{itm:hypergraph-cycle-theorem},
let $A$ be a uniformly random set of $n$ vertices, and let $B$
be a uniformly random loose Hamilton cycle on $V\left(H\right)\backslash A$.
Then a.a.s.{} $G_{A,B}\left(H\right)$ has minimum degree at least $\aham n$.}

\end{enumerate}\end{lem}
\begin{proof}
[Proof of \ref{lem:hypergraph-bipartite-degree-transfer}\oldref{itm:hypergraph-bipartite-matching-degree-transfer}]As in the preceding discussion, it is convenient to realize the uniform
distribution of $A$ and $B$ via a random ordering of $V\left(H\right)$.
Let 
\[
a^{1},\dots,a^{n},\, b_{1}^{1},\dots,b_{k-1}^{1},\, b_{1}^{2},\dots,b_{k-1}^{n}.
\]
be a uniformly random ordering of $V\left(H\right)$, with $A$, $b^i$ and
$B$ defined as before.

First, condition on some $\left(k-1\right)$-tuple $b\in B$, and imagine that the $a^i$ are chosen one-by-one. Given $b$ and $a^{1},\dots,a^{i}$, this means that $a^{i+1}$ has a uniformly random distribution from the remaining vertices. There are at least $\delta_{k-1}\left(H\right)$ choices of $a$ that will make $b\cup\left\{a\right\}$ an edge of $E\left(H\right)$, so for $i\le\alpha n/2$,
\[
\Pr\left(b\cup\left\{ a^{i+1}\right\} \in E\left(H\right)\mid a^{1},\dots,a^{i}\right)\ge\left(\delta_{k-1}\left(H\right)-i\right)/\left(kn\right)\ge\alpha/\left(2k\right).
\]
The degree of $b$ in $G_{A,B}\left(H\right)$ is the number of edges $b\cup\left\{ a^{i}\right\}$ in $E\left(H\right)$, which we have just shown stochastically dominates a $\Bin(\alpha n/2,\alpha/\left(2k\right))$ distribution.
By the Chernoff bound, $d_{G_{A,B}\left(H\right)}\left(b\right)\ge\alpha^{2}n/\left(8k\right)$
with probability $1-e^{-\Omega\left(n\right)}$. With the union bound, a.a.s.{} $d_{G_{A,B}\left(H\right)}\left(b\right)\ge\alpha^{2}n/\left(8k\right)$ for all $b$.

Now, instead condition on some $a\in A$ and imagine that the $\left(k-1\right)$-tuples $b^i$ are chosen one-by-one from the remaining vertices
(before choosing the rest of vertices of $A$). Note that there are at least $\delta_{1}\left(H\right)$ choices of a $\left(k-1\right)$-tuple $b$ such that $\left\{a\right\}\cup b\in E\left(H\right)$, and note that each $b_{j}^{i}$
shares at most ${kn \choose k-2}$ edges with $a$. Recall that $\delta_1\left(H\right)=\Omega\left(n^{k-1}\right)$, so  if $i\le2\sqrt{\amat}n$ for sufficiently small $\amat$ and large $n$,
then
\[
\Pr\left(\left\{ a\right\} \cup b^{i+1}\in E\left(H\right)\mid b^{1},\dots,b^{i}\right)\ge\left(\delta_{1}\left(H\right)-\left(k-1\right)i{kn \choose k-2}\right)\left.\vphantom{\sum}\right/{kn \choose k-1}\ge\sqrt{\amat}.
\]
By the same argument as in the previous paragraph, using the Chernoff bound and the union bound, a.a.s.{} each $d_{G_{A,B}\left(H\right)}\left(a\right)\ge\amat n$.
\end{proof}

\begin{proof}
[Proof of \ref{lem:hypergraph-bipartite-degree-transfer}\oldref{itm:hypergraph-bipartite-cycle-degree-transfer}]We give
essentially the same proof as for \ref{lem:hypergraph-bipartite-degree-transfer}\oldref{itm:hypergraph-bipartite-matching-degree-transfer}. As in the discussion preceding the lemma, choose a uniformly random ordering
\[
a^{1},\dots,a^{n},b_{0},\dots,b_{\left(k-2\right)n-1},
\]
and let $A$, $b^i$ and $B$ be defined as before.

With exactly the same proof as for \ref{lem:hypergraph-bipartite-degree-transfer}\oldref{itm:hypergraph-bipartite-matching-degree-transfer},
there is small $\aham$ such that a.a.s.{} $d_{G_{A,B}\left(H\right)}\left(b\right)\ge\aham n$
for all $b\in B$. Next, condition on some $a\in A$, and imagine the $b^i$ are chosen one-by-one
(before choosing the rest of vertices of $A$). Note  that the only intersection of $b^i$ with any of $b^0,\dots,b^{i-1}$ is the vertex $b_{\left(k-2\right)i}$, and the other vertices of $b^i$ are chosen uniformly at random from what remains. Also note that each $b_j$ shares at most ${\left(k-1\right)n \choose k-3}$ edges with both $a$ and $b_{\left(k-2\right)i}$. So, for small $\beta_k$, large $n$
and $1\le i\le2\sqrt{\aham}n$,
\[
\Pr\left(\left\{ a\right\} \cup b^{i}\in E\left(H\right)\mid b^{0},\dots,b^{i-1}\right)\ge\left(\delta_{2}\left(H\right)-\left(k-2\right)i{\left(k-1\right)n \choose k-3}\right)\left.\vphantom{\sum}\right/{ \left(k-1\right)n \choose k-2}\ge\sqrt{\aham}.
\]
By the same argument as before, a.a.s.{} each $d_{G_{A,B}\left(H\right)}\left(a\right)\ge\aham n$.
\end{proof}
We have established that if $\chyp\left(\alpha\right)$ 
is large enough then $G_{A,B}\left(H\right)$ is a.a.s.{} a bipartite graph with minimum degree $\a\left(\alpha\right)n$, and $G_{A,B}\left(R\right)$ contains a $\left(1-\xi\left(\a\left(\alpha\right)\right)\right)n$-edge sub-matching $G_{A,B}\left(Q\right)$ of the uniformly random perfect matching $G_{A,B}\left(\bar{Q}\right)$. \ref{lem:bipartite-plus-big-matching-perfect} then
ensures the existence of a perfect matching in $G_{A,B}\left(H\cup R\right)$. This corresponds to a perfect matching or loose Hamilton cycle in $H\cup R$.

\section{\label{sec:digraphs}Pancyclicity in dense digraphs}

In this section we prove \ref{lem:pseudorandom-pancyclic,thm:smoothed-pancyclic}. One motivation to consider pancyclicity instead of just Hamiltonicity is an observation by Bondy (see \cite{Bon75}),
that almost all known non-trivial conditions that ensure Hamiltonicity
also ensure pancyclicity. He even made an informal ``metaconjecture''
that this was always the case; our \ref{thm:smoothed-pancyclic} verifies
his metaconjecture in the setting of randomly perturbed dense graphs
and digraphs.

\ref{thm:smoothed-pancyclic} obviously implies \cite[Theorems~1 and~3]{BFM03}.
We do not fight very hard to optimize constants, but we note that if we make more careful calculations with our proof approach, then the resulting values of $\cdi\left(\alpha\right)$ seem to be better
than those found in \cite{BFM03}, for most values of $\alpha$.

We now turn to the proof of \ref{lem:pseudorandom-pancyclic}, which
will follow from the corresponding result for Hamiltonicity.
\begin{lem}
\label{lem:pseudorandom-hamiltonian}Let $D$ be a directed graph
with all in- and out- degrees at least $4k$, and suppose for every pair of
disjoint sets $A,B\subseteq V\left(D\right)$ with $\left|A\right|=\left|B\right|\ge k$,
there is an arc from $A$ to $B$. Then $D$ is Hamiltonian.\end{lem}

The idea of the proof is to start with a longest path $P$ and to manipulate
it into a cycle $C$ on the same vertex set. We will show that $D$
is strongly connected, so if $C$ were not Hamiltonian, there would
be an arc from $V\left(C\right)$ to its complement, which could be
combined with $C$ to give a longer path than $P$, contradicting
maximality. This type of argument goes back to the proof of Dirac's
theorem \cite[Theorem~3]{Dir52}. It also bears some resemblance to
the ``rotation-extension'' idea introduced by P\'osa in \cite{Pos76}, and
a variation for directed graphs by Frieze and Krivelevich \cite[Section~4.3]{FK05}.
\begin{proof}
[Proof of \ref{lem:pseudorandom-hamiltonian}]First we acknowledge
some immediate consequences of the condition on $D$. Note that if
$A$ and $B$ are disjoint sets with size at least $k$, then in fact
there are at least $\left|A\right|-k$ vertices of $A$ with an arc into $B$.
To see this, note that for any fewer number of such vertices in $A$,
we can delete those vertices and at least $k$ will remain, one of
which has an arc to $B$. Also, $D$ is strongly connected. To see
this, note that for any $v,w$, both of $N^{+}\left(v\right)$ and
$N^{-}\left(w\right)$ have size at least $4k>k$. If they intersect
then there is a length-2 path from $v$ to $w$; otherwise there must
be an arc from $N^{+}\left(v\right)$ to $N^{-}\left(w\right)$ giving
a length-3 path.

Let $P=u,\dots,w$ be a maximum-length directed path in $D$. We will
use the notation $v^{+}$ (respectively $v^{-}$) for the successor
(respectively predecessor) of a vertex $v$ on $P$, and also write
$U^{+},U^{-}$ for the set of successors or predecessors of a set
of vertices $U$.

By maximality, $N^{+}\left(w\right)\subset P$ and $N^{-}\left(u\right)\subset P$.
Let $U_{1}$ be the first $3k$ elements of $N^{-}\left(u\right)$
on $P$, and let $U_{2}$ be the last $k$ (note $U_{1}\cap U_{2}=\varnothing$).
Similarly let $W_{1}$ be the first $k$ and $W_{2}$ the last $3k$
elements of $N^{+}\left(w\right)$. We will now show that there is
a cycle on the vertex set $V\left(P\right)$.

First, consider the case where each vertex of $W_{1}$ precedes each
vertex of $U_{2}$. If $wu$ is in $D$ then we can immediately close
$P$ into a cycle. Otherwise, $\left|W_{1}^{-}\right|=\left|W_{1}\right|=\left|U_{2}^{+}\right|=\left|U_{2}\right|=k$,
so there is an arc $w_{1}u_{2}$ from $W_{1}^{-}$ to $U_{2}^{+}$.
This is enough to piece together a cycle on $V\left(P\right)$: start
at $u_{2}$ and move along $P$ to $w$, from where there is a shortcut
back to $w_{1}^{+}$. Now move along $P$ from $w_{1}^{+}$ to $u_{2}^{-}$,
from where we can jump back to $u$, then move along $P$ to $w_{1}$,
then jump to $u_{2}$. See \ref{fig:pseudorandom-hamiltonian-case-1}
for an illustration.

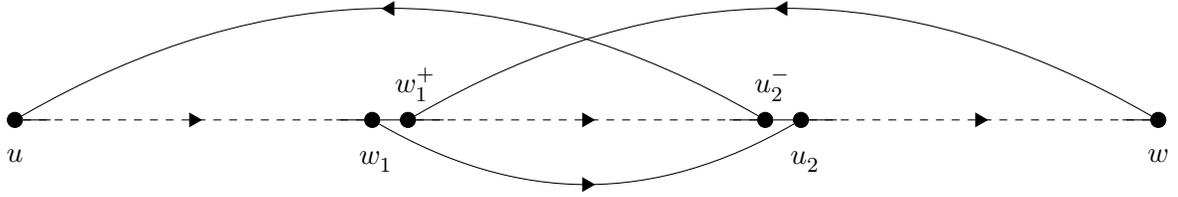
\begin{figure}[H]
\begin{center}
\hspace{0.5cm}
\begin{tikzpicture}[scale=0.95]

\node[vertex] (u) at  (0,0) {};
\node[vertex] (w1) at  (5,0) {};
\node[vertex] (w1p) at  (5.5,0) {};
\node[vertex] (u2m) at  (10.5,0) {};
\node[vertex] (u2) at (11,0) {};
\node[vertex] (w) at (16,0) {};

\draw (u) to (0.5,0);
\draw[arc,dashed] (u) to (w1);
\draw (4.5,0) to (w1);
\draw (w1) to (w1p);
\draw (w1p) to (6,0);
\draw[arc,dashed] (w1p) to (u2m);
\draw (10,0) to (u2m);
\draw (u2m) to (u2);
\draw (u2) to (11.5,0);
\draw[arc,dashed] (u2) to (w);
\draw (15.5,0) to (w);
\draw[arc] (w) to [bend right] (w1p);
\draw[arc] (u2m)  to [bend right] (u);
\draw[arc] (w1) to [bend right] (u2);

\node[larrow] at (5.25,1.56) {};
\node[larrow] at (10.75,1.56) {};
\node[rarrow] at (2.5,0) {};
\node[rarrow] at (8,0) {};
\node[rarrow] at (13.5,0) {};
\node[rarrow] at (8,-0.905) {};

\node at (5.1,-0.5){$w_1^{\phantom+}$};
\node at (5.6,0.5){$w_1^+$};
\node at (10.6,0.5){$u_2^-$};
\node at (11.1,-0.5){$u_2^{\phantom-}$};
\node at (0,-0.5){$u$};
\node at (16,-0.5){$w$};

\end{tikzpicture}
\hspace{0.5cm}
\end{center}

\protect\caption{\label{fig:pseudorandom-hamiltonian-case-1}The case where the vertices
of $W_{1}$ precede the vertices of $U_{2}$. The horizontal line
through the center is $P$; the broken lines indicate subpaths.}
\end{figure}

Otherwise, each vertex of $U_{1}$ precedes each vertex of $W_{2}$.
Let $U_{12}$ contain the $k$ elements of $U_{1}$ furthest down
the path. Note that there are at least $2k$ vertices of $P$ (e.g., vertices of $U_1\backslash U_{12}$) preceding all vertices in $U_{12}$. 
Let $U_{11}$ be the set of vertices among those first $2k$
vertices of $P$ which have an arc to $U_{12}^{+}$. By the discussion
at the beginning of the proof, $\left|U_{11}\right|\ge k$. Similarly,
let $W_{21}$ contain the $k$ elements of $W_{2}$ first appearing
on the path, and let $W_{22}$ be the set of at least $k$ vertices
among the last $2k$ on $P$ which have an arc from $W_{21}^{-}$.
By the condition on $D$, there is an arc $w_{22}u_{11}$ from $W_{22}^{-}$
to $U_{11}^{+}$. By definition, there is $u_{12}\in U_{12}^{+}$
and $w_{21}\in W_{21}^{-}$ such that the arcs $u_{11}^{-}u_{12}$
and $w_{21}w_{22}^{+}$ are in $D$. We can piece everything together
to get a cycle on the vertices of $P$: start at $u_{11}$, move along
$P$ until $u_{12}^{-}$, then jump back to $u$. Move along $P$
until $u_{11}^{-}$, then take the shortcut to $u_{12}$. Continue
along $P$ to $w_{21}$, jump to $w_{22}^{+}$, continue to $w$,
jump back to $w_{21}^{+}$, and continue to $w_{22}$. From here there
is a shortcut back to $u_{11}$. See \ref{fig:pseudorandom-hamiltonian-case-2}.

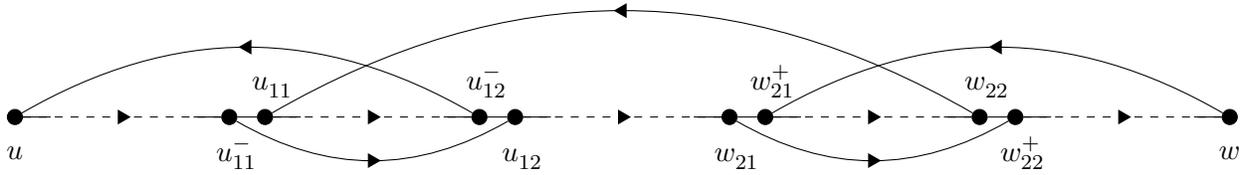
\begin{figure}[H]
\begin{center}
\hspace{0.5cm}
\begin{tikzpicture}[scale=0.95]

\node[vertex] (u) at  (0,0) {};
\node[vertex] (u11m) at  (3,0) {};
\node[vertex] (u11) at  (3.5,0) {};
\node[vertex] (u12m) at  (6.5,0) {};
\node[vertex] (u12) at  (7,0) {};
\node[vertex] (w21) at  (10,0) {};
\node[vertex] (w21p) at  (10.5,0) {};
\node[vertex] (w22) at  (13.5,0) {};
\node[vertex] (w22p) at  (14,0) {};
\node[vertex] (w) at  (17,0) {};

\draw (u) to (0.5,0);
\draw[arc,dashed] (u) to (u11m);
\draw (2.5,0) to (u11m);
\draw (u11m) to (u11);
\draw (u11) to (4,0);
\draw[arc,dashed] (u11) to (u12m);
\draw (u12m) to (6,0);
\draw (u12m) to (u12);
\draw (u12) to (7.5,0);
\draw[arc,dashed] (u12) to (w21);
\draw (9.5,0) to (w21);
\draw (w21) to (w21p);
\draw (w21p) to (11,0);
\draw[arc,dashed] (w21p) to (w22);
\draw (13,0) to (w22);
\draw (w22) to (w22p);
\draw (w22p) to (14.5,0);
\draw[arc,dashed] (w22p) to (w);
\draw (16.5,0) to (w);
\draw[arc] (u12m) to [bend right] (u);
\draw[arc] (u11m)  to [bend right] (u12);
\draw[arc] (w21) to [bend right] (w22p);
\draw[arc] (w) to [bend right] (w21p);
\draw[arc] (w22) to [bend right] (u11);

\node[larrow] at (3.25,0.98) {};
\node[larrow] at (13.75,0.98) {};
\node[larrow] at (8.5,1.49) {};
\node[rarrow] at (1.5,0) {};
\node[rarrow] at (5,0) {};
\node[rarrow] at (8.5,0) {};
\node[rarrow] at (12,0) {};
\node[rarrow] at (15.5,0) {};
\node[rarrow] at (5,-0.61) {};
\node[rarrow] at (12,-0.61) {};

\node at (0,-0.5){$u$};
\node at (3.1,-0.5){$u_{11}^-$};
\node at (3.6,0.5){$u_{11}^{\phantom-}$};
\node at (6.6,0.5){$u_{12}^-$};
\node at (7.1,-0.5){$u_{12}^{\phantom-}$};
\node at (10.1,-0.5){$w_{21}^{\phantom+}$};
\node at (10.6,0.5){$w_{21}^+$};
\node at (13.6,0.5){$w_{22}^{\phantom+}$};
\node at (14.1,-0.5){$w_{22}^+$};
\node at (17,-0.5){$w$};

\end{tikzpicture}
\hspace{0.5cm}
\end{center}

\protect\caption{\label{fig:pseudorandom-hamiltonian-case-2}The case where the vertices
of $U_{1}$ precede the vertices of $W_{2}$.}
\end{figure}

As outlined, the fact that $D$ is strongly connected, combined with
the fact that the vertex set of $P$ induces a cycle $C$, implies
that $C$ is a Hamilton cycle.
\end{proof}

We note that with some effort, the ideas in the proof of
\ref{lem:pseudorandom-hamiltonian} can be used directly to prove
 \ref{lem:pseudorandom-pancyclic} with a weaker degree condition.
We do not know whether the condition can be weakened all the way to
$4k$, as Bondy's metaconjecture would suggest. Also, the constants in \ref{lem:pseudorandom-pancyclic,lem:pseudorandom-hamiltonian}
can both be halved for the undirected case, just by simplifying the
main argument in the proof of \ref{lem:pseudorandom-hamiltonian}.

\begin{proof}
[Proof of \ref{lem:pseudorandom-pancyclic}]Fix a vertex $v$. Let
$U^{+}$ and $U^{-}$ be arbitrary disjoint $k$-subsets of $N^{+}\left(v\right)$
and $N^{-}\left(v\right)$ respectively. There is an arc from $U^{+}$
to $U^{-}$ which immediately gives a 3-cycle.

Next, let $W^{+}$ be the set of (fewer than $k$) vertices with no
arc from $U^{+}$, and similarly let $W^{-}$ be the set of vertices
with no arc into $U^{-}$. Now consider the induced digraph $D'$
obtained from $D$ by deleting $v$ and the vertices in $U^{+},U^{-},W^{+},W^{-}$.
Since we have removed fewer than $4k$ vertices, $D'$ satisfies the
conditions of \ref{lem:pseudorandom-hamiltonian} so has a Hamilton
cycle. In particular, for every $\ell$ satisfying $4\le\ell\le n-4k$,
there is a path $P_{\ell}=u_{\ell},\dots, w_{\ell}$ in $D'$ of length
$\ell-4$. By construction, there is an arc from $U^{+}$ to $u_{\ell}$
and from $w_{\ell}$ to $U^{-}$, which we can combine with arcs to
and from $v$ to get a cycle of length $\ell$.

Finally, for every $\ell>n-4k$, arbitrarily delete vertices from
$D$ to obtain an induced digraph $D''$ with $\ell$ vertices which
satisfies the conditions of \ref{lem:pseudorandom-hamiltonian}. Since
$D''$ has a Hamilton cycle, $D$ has a cycle of length $\ell$.
\end{proof}

\begin{proof}
[Proof of \ref{thm:smoothed-pancyclic}]In view of the discussion
in the introduction, we assume each possible arc is present in $R$
with probability $p=\cdi n/\left(2{n \choose 2}\right)\ge\cdi/n$ independently. (Recall that digraphs are allowed to have 2-cycles).

If $A,B\subseteq V\left(D\right)$ are disjoint sets with $\left|A\right|=\left|B\right|=\alpha n/8$,
the probability that there are no arcs from $A$ to $B$ in $D\cup R$
is at most
\[
\left(1-p\right)^{\left(\alpha n/8\right)^{2}}\le e^{-p\alpha^{2}n^{2}/64}\le e^{-\cdi\alpha^{2}n/64}.
\]
The number of choices of such pairs of disjoint sets $A,B$ is at most $2^{2n}$. By the union bound,
the probability that $D\cup R$ does not satisfy the condition of
\ref{lem:pseudorandom-hamiltonian} is at most $2^{2n}e^{-\cdi\alpha^{2}n/64}$, which 
converges to zero for sufficiently large $\cdi$. Therefore the digraph $D\cup R$ is a.a.s.{} 
pancyclic by \ref{lem:pseudorandom-hamiltonian}.\end{proof}

\section{\label{sec:tournaments}Hamilton cycles in tournaments}

There are several seemingly different conditions that are equivalent
to Hamiltonicity for tournaments (see \cite[Chapters~2-3]{Moo68}).
A tournament is Hamiltonian if and only if it is irreducible (cannot
be divided into two parts with all arcs between the two parts
in the same direction), if and only if it is strongly connected (has
a directed path from every vertex to every other), if and only if
it is pancyclic (contains cycles of all lengths). All tournaments
have a Hamilton path, and it was first proved by Moon and Moser in \cite{MM62} that
a uniformly random tournament is a.a.s.{} irreducible, hence Hamiltonian.

More recently, K{\"u}hn, Lapinskas, Osthus and Patel \cite{KLOP14} proved in  that if a tournament
is $t$-strongly connected (it remains strongly connected after the
deletion of $t-1$ vertices), then it has $\Omega\left(\sqrt{t}/\log t\right)$
arc-disjoint Hamilton cycles. (This was improved to $\Omega\left(\sqrt{t}\right)$
by Pokrovskiy \cite{Pok14}). Therefore, to show a tournament has $q$ arc-disjoint Hamilton cycles for any $q=O\left(1\right)$, it suffices to show that the tournament is $t$-strongly connected for any fixed $t$. In particular, it is not difficult to show that a random tournament is a.a.s.{} $t$-strongly connected for fixed $t$, which motivates \ref{thm:tournament}.

Before we proceed to the proof, we first explain why \ref{thm:tournament}
is sharp. The ``obvious'' worst case for $T$ is a \emph{transitive}
tournament (corresponding to a linear order on the vertices). In this
case, a superlinear number of random edges must be flipped in order to a.a.s.{}
flip one of the arcs pointing away from the least element of the linear
order. Actually, {\L}uczak, Ruci{\'n}ski and Gruszka \cite{LRG96} have already studied the model where random edges are flipped in a transitive
tournament, by analogy to
the evolution of the random graph.

More generally, consider a ``transitive cluster-tournament'' $T$
on $n=r\left(2d+1\right)$ vertices defined as follows. Let $R$ be a regular tournament
on $2d+1$ vertices (this means every vertex has indegree and outdegree $d$). To construct $T$, start with $r$ disjoint copies
$R_{1},\dots,R_{r}$ of $R$, then put an arc from $v$ to $w$ for
every $v\in R_{i}$, $w\in R_{j}$ with $i<j$. In order for the perturbed
tournament $P$ to be Hamiltonian, there must be an arc entering $R_{1}$,
so one of the $O\left(n\left(d+1\right)\right)$ arcs exiting $R_{1}$
must be changed. This will not happen a.a.s.{} unless $m=\omega\left(n/\left(d+1\right)\right)$.

We now prove \ref{thm:tournament}. In accordance with the discussion
in the introduction, we will work with the model where each edge is
flipped with probability $p$ independently, where $2p{n \choose 2}=m$.
(Designating an edge for resampling with probability $2p$ is the
same as flipping it with probability $p$). Note that $p\le1/2$ and in particular $p\le1-p$.

Fix $t$, and let $T$ be a tournament with $n$ vertices
and all in- and out- degrees at least $d$. As described above, flip each edge with probability $p=\omega\left(1/\left(n\left(d+1\right)\right)\right)$, to obtain a perturbed tournament $P$. We will prove that $P$ is a.a.s.{} $t$-strongly connected.

The idea of the proof is to choose a set $S$ of $t$ vertices with a
large indegree and outdegree, then to show that with high probability
almost every vertex has many paths to and from each vertex in $S$.
The probability that a vertex $v$ has paths to and from from $S$ is smallest
if $v$ has small indegree, so we need to show that not many vertices
can have small indegree.
\begin{lem}
\label{lem:extreme-indegree-bound}In any tournament, there are fewer
than $k$ vertices with indegree (respectively outdegree) less than
$\left(k-1\right)/2$.\end{lem}
\begin{proof}
The sum of indegrees (respectively outdegrees) of a tournament on
$k$ vertices is ${k \choose 2}$, because each arc contributes 1
to this sum. Therefore in every set of $k$ vertices of a tournament,
there is a vertex of outdegree (indegree) at least $\left(k-1\right)/2$
in the induced sub-tournament.
\end{proof}
A consequence of \ref{lem:extreme-indegree-bound} is that the set
of all vertices with indegree (respectively outdegree) less than $n/6$
in $T$, has size smaller than $n/3$. So, there are at least $n/3$ vertices
whose indegree and outdegree in $T$ are both at least $n/6$. We can therefore choose
a set $S$ of $t$ such vertices.

Now, we prove that the random perturbation typically does not reduce
the in- and out- degrees very much. 
\begin{lem}
\label{lem:preserve-degree}If a vertex $v$ has outdegree (respectively
indegree) $k$ in $T$, then it has outdegree (respectively indegree)
at least $k/3$ in the randomly perturbed tournament $P$, with probability $1-o\left(e^{-k}\right)-o\left(1/n\right)$
(uniformly over $k$).\end{lem}
\begin{proof}
We only prove the statement where $v$ has outdegree $k$; the indegree
case is identical.

There are $k$ arcs pointing away from $v$ in $T$. If $p<1/\sqrt{n}$ then the probability more than  $2k/3$ of those arcs are changed by the perturbation is at most $2^kp^{2k/3}=e^{-k\Omega\left(\log n\right)}=o\left(e^{-k}\right)$. If $k\ge\sqrt{n}/2$ then by the Chernoff bound and the fact that $p\le1/2$, the probability more than $2k/3$ arcs are changed is $e^{-\Omega\left(\sqrt{n}\right)}=o\left(1/n\right)$. In both of these cases, $k/3$ of the original out-neighbours survive
the perturbation with the required probability.

The remaining case is where $p\ge1/\sqrt{n}$ and $k<\sqrt{n}/2$.
In this case there is a set of $n/2$ arcs pointing towards $v$ in
$T$. The Chernoff bound says that the probability less than $k/3$
of these arcs are changed is $e^{-\Omega\left(\sqrt{n}\right)}=o\left(1/n\right)$.
That is, with the required probability, $k/3$ new out-neighbours
are added by the perturbation.\end{proof}
\begin{lem}
\label{lem:path-from-home}Suppose $w$ has outdegree (respectively
indegree) at least $n/6$ in $T$, and $v$ is a vertex different
from $w$ with indegree (respectively outdegree) at least $k$ in
$T$. Then with probability $1-o\left(e^{-k/\left(d+1\right)}\right)-o\left(1/n\right)$
(uniformly over $k$), there are $t'=3t$ internally vertex-disjoint
paths of length at most 3 from $w$ to $v$ (respectively from $v$
to $w$) in the randomly perturbed tournament $P$.\end{lem}
\begin{proof}
We will only prove the statement where $v$ has indegree at least
$k$; the outdegree case is identical. By independence, we can condition
on the outcome of the perturbation on individual arcs. Condition on
the outcome for all arcs adjacent to $w$, and let $N_{P}^{+}\left(w\right)$
be the set of vertices to which there is an arc from $w$ in $P$.
By \ref{lem:preserve-degree}, we can assume $\left|N_{P}^{+}\left(w\right)\right|\ge n/18$.

We first prove the lemma for the case where $k\le6t'$. There are
at least $n'=n/18-1$ arcs between $N_{P}^{+}\left(w\right)$ and
$v$, each of which will be pointing towards $v$ in $P$ with probability
at least $p$ independently. By the Chernoff bound, the probability
less than $t'$ arcs will point from $N_{P}^{+}\left(w\right)$ to
$v$ in $P$ is $e^{-\Omega\left(np\right)}=o\left(e^{-k/\left(d+1\right)}\right)$.
So, with the required probability there are $t'$ suitable length-2
paths from $w$ to $v$.

We can now assume $k>6t'$. Condition on the result of the perturbation
for the arcs adjacent to $v$ (in addition to the arcs we have conditioned
on so far). Let $N_{P}^{-}\left(v\right)$ be the set of vertices
from which there is an edge into $v$ in $P$; by \ref{lem:preserve-degree},
we can assume $\left|N_{P}^{-}\left(v\right)\right|\ge k/3$.

Now, if $\left|N_{P}^{+}\left(w\right)\cap N_{P}^{-}\left(v\right)\right|\ge t'$
then there are $t'$ disjoint length-2 paths from $w$ to $v$ and
we are done. So we can assume $U^{+}=N_{P}^{+}\left(w\right)\backslash\left(N_{P}^{-}\left(v\right)\cup\left\{ v\right\} \right)$
has at least $n'=n/18-t'-1$ vertices, and $U^{-}=N_{P}^{-}\left(v\right)\backslash\left\{ w\right\} $
has at least $k'=k/3-1$ vertices (note $k'\ge 2t'$ by assumption).

Now, we would like to show that with the required probability there
is a set of $t'$ independent arcs from $U^{+}$ into $U^{-}$ in
$P$, which will give $t'$ suitable length-3 paths. Partition $U^{+}$(respectively
$U^{-}$) into subsets $U_{1}^{+},\dots,U_{t'}^{+}$ of size at least
$n'/\left(2t'\right)$ (respectively, subsets $U_{1}^{-},\dots,U_{t'}^{-}$
of size at least $k'/\left(2t'\right)$). Recall that $1-p\ge p$, so for each $i$, the probability
that there is no arc from $U_{i}^{+}$ into $U_{i}^{-}$ after the
perturbation is at most 
\[
\left(1-p\right)^{n'k'/\left(2t'\right)^{2}}=e^{-\Omega\left(npk\right)}=o\left(e^{-k/\left(d+1\right)}\right).
\]
We conclude that with the required probability, there is a set of $t'$ suitable independent
arcs, each between a pair $U_{i}^{+},U_{i}^{-}$.\end{proof}
\begin{lem}
\label{lem:lots-of-paths}Fix some $w\in S$. In the randomly perturbed tournament $P$, there are a.a.s.{}
$t$ internally vertex-disjoint paths from $w$ to each other vertex
(respectively, from each other vertex to $w$).\end{lem}
\begin{proof}
We only prove there are paths from $w$ to each other vertex; the
reverse case is identical. If there are $3t$ internally vertex-disjoint
paths from $w$ to $v$ in $P$ of length at most 3, then we say $v$
is \emph{safe}. It follows from \ref{lem:path-from-home} that a vertex
with indegree $k$ is safe with probability at least $1-f\left(n\right)e^{-k/\left(d+1\right)}+f\left(n\right)/n$,
for some $f\left(n\right)=o\left(1\right)$.

By \ref{lem:extreme-indegree-bound}, there are at most $2d+1$ vertices
with indegree $d$ in $T$, and the vertex with the $2k$-th smallest
indegree has indegree at least $k-1$. Let $Q$ be the set of non-safe
vertices, and note 
\begin{align*}
\E\left[\left|Q\right|\right] & \le\left(2d+1\right)f\left(n\right)e^{-d/\left(d+1\right)}+\sum_{k=d+1}^{n}2f\left(n\right)e^{-k/\left(d+1\right)}+nf\left(n\right)/n\\
 & =f\left(n\right)\left(O\left(d+1\right)+\frac{1}{1-e^{-1/\left(d+1\right)}}\right)\\
 & =f\left(n\right)O\left(d+1\right).
\end{align*}
(We have used the geometric series formula and the inequality $1-e^{-x}\ge\left(x\land1\right)/2$
for positive $x$).

By Markov's inequality, a.a.s.{} $\left|Q\right|\le\sqrt{f\left(n\right)}O\left(d+1\right)=o\left(d+1\right)$.
If $d\le12t$ then $\left|Q\right|=0$ for large $n$ and we are done. Otherwise, $d>12t$, and we will show that there are a.a.s.{} suitable paths of length at most 4 from $w$ to each other vertex.

By \ref{lem:preserve-degree} the perturbation a.a.s.{} reduces the degree of each vertex by a factor of at most 3 (the probability of this not occurring is $o\left(\sum_{k=0}^\infty e^{-k}\right)+o\left(1\right)=o\left(1\right)$), so every vertex $v\in Q$ a.a.s.{} has indegree at least $d/3$. We have $\left|Q\right|<d/12$ for large $n$, and we are assuming $d>12t$, so for every vertex $v$ we can choose $d/3-d/12-1=d/4-1\ge 3t$ safe in-neighbours $v_{1},\dots,v_{3t}$ different from $w$.

Now, fix some $v$ and consider a maximum-size collection $M$ of internally vertex-disjoint paths of length at most 4 from $w$ to $v$ in $P$. For the purpose of contradiction assume that $\left|M\right|<t$. There are fewer than $3t$ internal vertices in the paths in $M$, so there is some $v_{i}$ not in any path of $M$. Then, there are $3t$ internally vertex-disjoint paths from $w$ to $v_{i}$, so one of these does not contain any
internal vertex of a path in $M$. It follows that $w,\dots,v_{i},v$ is a path from $w$ to $v$ which is
internally vertex-disjoint from every path in $M$, contradicting maximality.
\end{proof}
It follows from \ref{lem:lots-of-paths} that in $P$, a.a.s.{} every vertex
outside $S$ has $t$ internally vertex-disjoint paths to and from every vertex in $S$.
If $P$ has this property and we delete $t-1$ vertices, then there is at least one vertex $w$
of $S$ remaining, and $w$ has least one path to and from every other
vertex. That is, $P$ is a.a.s.{} $t$-strongly connected. (In fact, we have also
proved that $P$ a.a.s.{} has diameter at most 8, since there is a path of length at most 4 between $w$ and any other vertex). Recall that a $t$-strongly connected tournament has $\Omega\left(\sqrt t\right)$ arc-disjoint Hamilton cycles, so we can conclude that $P$ a.a.s.{} has $q$ arc-disjoint Hamilton cycles for any $q=O\left(1\right)$.

\section{Concluding Remarks}

We have determined the amount of random perturbation typically required to make
a tournament, dense digraph or dense uniform hypergraph Hamiltonian. In
the process, we have proved a general lemma about pancyclicity in
highly connected digraphs, and demonstrated an interesting application
of the Szemer\'edi regularity lemma. In the hypergraph setting, there are some important questions
this paper leaves open.

First, we have only studied loose Hamiltonicity. The other most popular
notion of a hypergraph cycle is a \emph{tight} cycle, in which every
consecutive pair of edges in the cycle intersects in $k-1$ vertices.
More generally, an $\ell$-cycle has consecutive edges intersecting
in $k-\ell$ vertices. Also, we have only studied hypergraphs with
high $\left(k-1\right)$-degree, which is the strongest density assumption
we could make. There are a large variety of Dirac-type theorems for
different types of minimum degree and different types of cycles (see
\cite{RR10} for a survey), which would suggest that similar random
perturbation results are possible in these settings.

\medskip

\noindent{\bf Acknowledgement.} Parts of this work were carried out when the first author visited the Institute for Mathematical Research (FIM) of ETH Z\"urich, and also when the third author visited the School of Mathematical Sciences of Tel Aviv University, Israel. We would like to thank both institutions for their hospitality and for creating a stimulating research environment. We also thank the anonymous referee for their thorough proofreading.

\bibliographystyle{amsplain}
\providecommand{\bysame}{\leavevmode\hbox to3em{\hrulefill}\thinspace}
\providecommand{\MR}{\relax\ifhmode\unskip\space\fi MR }
\providecommand{\MRhref}[2]{%
  \href{http://www.ams.org/mathscinet-getitem?mr=#1}{#2}
}
\providecommand{\href}[2]{#2}

\end{document}